\documentclass{article}

\usepackage[T1]{fontenc}

\usepackage{quiver}

\usepackage{amsthm, amsmath, amssymb}

\usepackage{tabularx}
\usepackage{Baskervaldx, mathpazo}

\usepackage[backref=true, backend=bibtex]{biblatex}
\DefineBibliographyStrings{english}{%
	page             = {},
	pages            = {},
}

\usepackage{enumerate}
\usepackage{tikz}
\usetikzlibrary{arrows,decorations.markings,positioning}
\tikzset{->-/.style={decoration={
			markings,
			mark=at position .5 with {\arrow{>}}},postaction={decorate}}}
\usepackage{tikz-cd}
\usepackage{eucal}
\usepackage{quiver}

\newtheorem{thm}{Theorem}[section]
\newtheorem{cor}[thm]{Corollary}
\newtheorem{prop}[thm]{Proposition}
\newtheorem{lem}[thm]{Lemma}

\theoremstyle{definition}
\newtheorem{defn}[thm]{Definition}

\newtheorem{exam}[thm]{Example}
\newtheorem{rmk}[thm]{Remark}
\theoremstyle{remark}

\newcommand{\Set}{\mathcal{S}et}

\DeclareMathOperator{\Fun}{Fun}

\newcommand{\Finp}{\mathcal{F}in_\ast}
\newcommand{\Setp}{\Set_\ast}
\newcommand{\Fin}{\mathcal{F}in}

\newcommand{\im}{\mathrm{im}}

\newcommand{\yo}{\text{\usefont{U}{min}{m}{n}\symbol{'210}}}
\DeclareFontFamily{U}{min}{}
\DeclareFontShape{U}{min}{m}{n}{<-> udmj30}{}

\newcommand{\op}{\mathrm{op}}

\newcommand{\ulp}[1]{\underline{#1}_+}
\newcommand{\ul}[1]{\underline{#1}}

\newcommand{\KK}{\mathbb{K}}

\renewcommand{\epsilon}{\varepsilon}
\newcommand{\fun}{\mathbb{F}_1}

\newcommand{\Part}{\mathtt{Part}}
\newcommand{\Dynk}{\mathtt{Dynk}}

\usepackage{hyperref}
\hypersetup{
	colorlinks   = true, 
	urlcolor     = black, 
	linkcolor    = blue, 
	citecolor   = black 
}
\title{Dynkin Systems and the One-Point Geometry}
\author{Jonathan Beardsley}

\begin{document}
	
	\maketitle
	
	\begin{abstract}
	In this note I demonstrate that the collection of Dynkin systems on finite sets assembles into a Connes-Consani $\fun$-module, with the collection of partitions of finite sets as a sub-module. The underlying simplicial set of this $\fun$-module is shown to be isomorphic to the delooping of the Krasner hyperfield $\mathbb{K}$, where $1+1=\{0,1\}$. The face and degeneracy maps of the underlying simplicial set of the $\fun$-module of partitions correspond to merging partition blocks and introducing singleton blocks, respectively. I also show that the $\fun$-module of partitions cannot correspond to a set with a binary operation (even partially defined or multivalued) under the ``Eilenberg-MacLane'' embedding. These results imply that the $n$-fold sum of the Dynkin $\fun$-module with itself is isomorphic to the $\fun$-module of the discrete projective geometry on $n$ points.  
	\end{abstract}

    \tableofcontents 
	
	\section{Introduction}
	This paper describes a novel relationship between four structures: Dynkin systems, or quantum mechanical $\sigma$-algebras \cite{derrwilliamson-dynkinsystems,suppes-probabilisticquantum}; discrete, or \textit{thin}, projective geometries in which a line may contain two or fewer points \cite[2.1.8]{faure-frolicher}  \cite{tits-F1}; the Krasner hyperfield $\KK=\{0,1\}$, in which addition is multivalued \cite{krasner-krasnerhyperfield, connesconsani-hyperringofadeles}; simplicial sets and Segal $\Gamma$-sets, especially in their role as models for deloopings, or classifying spaces, of algebraic structures \cite{segal, goerssjardine}. The last three items above have been explored in recent work of Connes and Consani suggesting the category of $\Gamma$-sets, or functors $X\colon \Finp\to\Setp$ from pointed finite sets to all pointed sets, as a model for algebra and geometry over a ``field with one element'' \cite{congammasets,congromovnorm,conabsoluteAG}. Work of myself and Nakamura has also strengthened, and made more explicit, the connection between discrete projective geometries and $\Gamma$-sets \cite{beardsleynakamura}. Before stating the main results of this paper, I will give brief summaries of each of the above notions.
	
	\textbf{Dynkin systems and pre-Dynkin systems} first arose in work of Dynkin on Markov systems and probability theory, where they were called $\lambda$-algebras \cite[1.1]{dynkinTheory}. Given a set $S$, a Dynkin system on $S$ is a family of subsets that contains the empty set, is closed under complementation, and is closed under countable unions of mutually disjoint sets. A pre-Dynkin system is the same except that it is closed only under \textit{finite} unions of mutually disjoint sets. Dynkin uses the fact that any Dynkin system which is additionally closed under arbitrary intersection is a $\sigma$-algebra, and this has become a standard tool in measure theory. Almost contemporaneously, Suppes proposed that Dynkin systems and pre-Dynkin systems, which he called quantum mechanical $\sigma$-algebras and quantum mechanical algebras respectively, should replace $\sigma$-algebras as the set-theoretic basis for measurable spaces in quantum mechanics \cite{suppes-probabilisticquantum}. This idea was also taken up by Gudder, who referred to them as $\sigma$-classes \cite{gudder-extensionofmeasuretheory,gudder-quantumprobabilityspaces}. They are also called (pre-)Dynkin systems, and used in probabilistic applications, in \cite{jun-tighterbounds, derrwilliamson-dynkinsystems, schurzleitgeb-finitistic}.  Since then, Dynkin systems and pre-Dynkin systems have appeared in a number of guises in probability theory, combinatorics and mathematical physics. The list of terminologies for them known to me, in addition to those already mentioned, are the following: additive classes \cite{raorao-charges}; concrete logics \cite{ovchinnikov-concretelogics, desimone-extendingstates}; partial fields \cite{godowski-varieties}; semi-algebras \cite{khrennikov-interpretations}; set-representable orthomodular posets \cite{ptak-booleanorthomodular}; specific sets \cite{dorninger-booleanorthomodular}; generalized fields of events \cite{dorfer-numericaleventspaces}; $\mathrm{d}$-systems \cite{williams-martingales}; $\lambda$-classes \cite{chow-probability}; $\lambda$-fields \cite{fine-probability}. 
	
	\textbf{Combinatorial projective geometries} are sets with some collection of subsets thought of as ``lines'' which satisfy axioms analogous to the usual rules of classical projective geometry \cite{faure-frolicher}. Tits pointed out that many of the geometric counting formulae for projective geometry have combinatorial analogues when one replaces the characteristic $p$ with the number $1$ \cite{tits-F1}. This led Tits to suggest that these can be thought of as ``projective geometries over a field with one element.'' Nakamura and I have shown that the category of combinatorial projective geometries can be embedded in a certain category of $\fun$-modules which we review later in this introduction \cite{beardsleynakamura}. This paper partially arose from an attempt to determine which $\fun$-modules correspond to the \textit{discrete} geometries. The discrete geometry on a set $S$ is the one in which there is a unique line through every pair of points, and no line contains more than those two points. When one attempts to do ``projective geometry'' with such a thing, one begins to see the combinatorial shadow that Tits observed.
	
	\textbf{The Krasner hyperfield} was originally introduced in \cite{krasner-krasnerhyperfield} in the context of $p$-adic analysis and the study of valued fields. It is the set $\KK=\{0,1\}$ equipped with a hyperoperation in which $0$ is the additive identity and $1+1=\{0,1\}$. The tropical hyperfield of tropical geometry can also be realized as an extension of $\KK$ \cite[Example 2.7]{maxwell-smith-tropicalextensions} and it has also appeared in matroid theory \cite{baker-lorscheid-foundations} \cite[Remark 3.19]{baker-bowler-matroids}. It can be useful to note that $\KK$ naturally appears as the set-theoretic quotient $\mathbb{C}/\mathbb{C}^\ast$ and so may be thought of as a kind of augmented projective point (this is in fact true if one replaces $\mathbb{C}$ by any field with at least three elements).
	
	\textbf{Simplicial sets and $\Gamma$-sets} are functors $\Delta^{\op}\to\Set$ and $\Gamma^{\op}\to\Set$ respectively, where $\Delta$ is the category of linearly ordered finite sets with non-decreasing functions and $\Gamma^{\op}$ is equivalent to the category of finite pointed sets with pointed functions. The former serve as a combinatorial foundation for homotopy theory \cite{goerssjardine} and higher category theory \cite{joyalapps,htt}. The latter have commutative monoids as a full subcategory and are the \textit{discrete} objects in Segal's approach to algebraic $K$-theory and stable homotopy theory \cite{segal}. It is also clear from \cite{segal} that every $\Gamma$-set has a ``delooping'' which is a simplicial set that, in good cases, plays the role of a classifying space for the algebraic structure encoded by that $\Gamma$-set. $\Gamma$-sets will play a central role in what follows, a description of which we turn to now.
	
	The setting for much of the work in this paper is the approach to algebra and geometry in ``characteristic one'' developed by Connes and Consani \cite{congammasets,congromovnorm,conabsoluteAG}. They refer to pointed $\Gamma$-sets, i.e.~functors which take $\{\ast\}$ to $\{\ast\}$, as $\mathfrak{s}$-modules or $\mathbb{S}$-modules, with the understanding that these are a model for $\fun$-modules. Given the connections to Arakelov geometry described in the preceding references, along with my recent work with Nakamura connecting them to combinatorics \cite{beardsleynakamura}, I will take the liberty of referring to pointed $\Gamma$-sets as $\fun$-modules throughout this work.
	
	\subsubsection*{Organization of the Paper}
	
	In Section \ref{sec:background}, we review the basic definitions of Connes-Consani $\fun$-modules, the functors of Dynkin systems and partitions, and the relationship between the plasmas of \cite{beardsleynakamura} and $\fun$-modules. We give several examples of plasmas, including the Krasner hyperfield $\KK$ and combinatorial projective geometries.
	
	In Section \ref{sec:plasmasofgeometries} we prove that the fully faithful embedding from combinatorial projective geometries to plasmas takes the discrete geometry on $n$ points to the $n$-fold coproduct $\vee_k\KK$ (Theorem \ref{thm:ProjisK}). We also show that plasma morphisms from the so-called power set plasma $\mathcal{P}(n)$ to $\vee_k\KK$ are in bijection with $k$-fold coproduct of plasma morphisms from $\mathcal{P}(n)$ to $\KK$ (Proposition \ref{prop:P(n)toVKfactorsthroughsummand}). This is essential to showing that the $\fun$-module of the $n$-point discrete geometry is the $n$-fold coproduct of the Dynkin $\fun$-module with itself.
	
	In Section \ref{sec:Dynkinsaskernels} we introduce \textit{KZ-systems} which, for pointed sets, are minimal data for generating all pre-Dynkin systems (Theorem \ref{thm:KZ to prelambda}). It follows from this that plasma morphisms from the power set plasma of any set $S$ into $\KK$ biject with pre-Dynkin systems on $S$ (Corollary \ref{cor:dynkinsystemsarekernels}). We then prove naturality of this isomorphism which allows us to identify the $\fun$-module of the discrete projective geometry on $n$ points with the $n$-fold coproduct of the Dynkin $\fun$-module with itself and the resulting computation of the delooping of $\KK$ (Theorem \ref{thm:F1modofprojisDynk} and Corollary \ref{cor:deloopingofKisDynk}).
	
	In Section \ref{sec:F1modofPartitions}, we show that the $\fun$-module of partitions (functorially) includes into the $\fun$-module of Dynkin systems (Theorem \ref{thm:partsubdynk}). We also explain why the $\fun$-module of partitions cannot be the $\fun$-module of any plasma. Its data is contained in strictly higher degrees.
	
	Finally, in Section \ref{sec:Dynkin Systems as Simplices}, we explicitly describe the simplices of the delooping of the Dynkin $\fun$-module in dimensions 0, 1 and 2, and give examples of some of the 3-simplices. We include discussion of the multiple ways of interpreting this ``geometric'' structure.

		\subsubsection*{Acknowledgements}
	Thanks to Bastiaan Cnossen, Sonja Farr, Landon Fox, Kiran Luecke, Joe Moeller, So Nakamura, Eric Peterson, Manny Reyes, Chris Rogers and David W\"arn for helpful conversations regarding this work. Thanks especially to Sonja Farr, who asked me if I had an explicit description of the delooping of $\KK$, the simplest non-trivial set with a multivalued addition. That was the spark that caused this paper to exist. Also thank you to Peter Taylor for computing the number of Dynkin systems on finite sets in low dimension. My computations had led to me the sequence $1$, $2$, $5$, $19$, $137$, $3708$, but this sequence did not exist on the OEIS when I started this paper. The only place this sequence appeared on the internet at the time, to my knowledge, was Peter Taylor's MathOverflow answer \cite{taylor-dynkinsystemsMO}. This allowed me to conjecture, and then prove, that the delooping of $\KK$ is also counting Dynkin systems. Thanks also to Martin Rubey for asking the question that Peter answered. This work was partially supported by the Simons Foundation, Award ID \#853272.
	
	\section{Background}\label{sec:background}

    As described in the introduction, this work mostly takes place in the context of algebra over $\fun$ suggested by Connes and Consani in \cite{congammasets}. 

    \subsection{Connes-Consani $\fun$-modules}

    The model of $\fun$-modules introduced by Connes and Consani is based on functors from finite pointed sets to all pointed sets. They have primarily used this category to give new foundations for Arakelov geometry, but in this work we will take advantage of their more combinatorial aspects.

    \begin{defn} The categorical foundations of this work depend on the following elementary definitions.
    \begin{enumerate}
        \item Write $\Setp$ for the category of pointed sets and functions between them that preserve the basepoint. We will typically write $\ast$ for a generic base point and $\{\ast\}$ for a generic one-element pointed set.
        
        \item Write $\Finp$ for the full subcategory of $\Setp$ spanned by the sets $\ulp{n}=\{0,1,2,\ldots,n\}$ where $0$ is taken to be the basepoint. When we wish to indicate the \textit{unpointed} set with $n$ elements, we write $\ul{n}=\{1,2,\ldots,n\}$.

        \item Both $\Setp$ and $\Finp$ have a symmetric monoidal structure given by the smash product $\wedge$, where $A\wedge B=(A\times B)/((a,\ast)\sim (\ast,b))$. The inclusion functor $\Finp\hookrightarrow\Setp$ is evidently symmetric monoidal with respect to this structure.
        \end{enumerate}
    \end{defn}

    \begin{defn}
        Write $\mathrm{Mod}_{\fun}$ for $\Fun_\ast(\Finp,\Setp)$ of pointed functors from $\Finp$ to $\Setp$. We denote the inclusion functor $\Finp\subseteq\Setp$ by $\fun$. This category admits the \textit{Day convolution} symmetric monoidal structure with respect to the smash product \cite{day-closedcategoriesoffunctors,kelly-im-universalconvolution}. Moreover, $\fun$ is the monoidal unit for this symmetric monoidal structure.
    \end{defn}
	
	\subsection{Functors of Partitions and Dynkin Systems}

    For a set $S$, $\mathcal{P}(S)$ will always denote the power set of $S$. The symbol $\varnothing$ will exclusively refer to the empty set \textit{as an element of a power set}. Whenever we want to talk about the empty set in any other context we will write $\{~~\}$. In the case that we are working both with a power set $\mathcal{P}(S)$ and a double power set $\mathcal{P}(\mathcal{P}(S))$ we will write $\varnothing $ for the element of the former and $\{~~\}$ for the element of the latter.
    
	\begin{defn}\label{def:dynkinsystems}
		Let $S$ be a set.
		
		\begin{enumerate}
			\item Define $\Dynk^{pre}(S)$ to be the set of families of subsets $X\subseteq\mathcal{P}(S)$ that satisfy the following conditions:
			\begin{enumerate}
				\item $\varnothing\in X$.
				\item If $A\in X$ then $S\setminus A\in X$.
				\item If $A,B\in X$ and $A\cap B=\varnothing$ then $A\cup B\in X$.
			\end{enumerate} We refer to $\Dynk^{pre}(S)$ as the set of pre-Dynkin systems on $S$. If $X$ satisfies the additional following condition then we call it a Dynkin system:
			\begin{enumerate}
				\item[(d)] If $\{A_i\}_{i\in\mathbb{N}}$ is a countable family of pairwise disjoint subsets and $A_i\in X$ for all $i$ then $\cup_iA_i\in X$.
			\end{enumerate}
			We write $\Dynk(S)$ for the set of all Dynkin systems on $S$.
			\item Define $\Part(S)$ to be the set of families $X\subseteq\mathcal{P}(S)$ such that $\cup_{A\in X}A=S$, $A\cap B=\varnothing$ whenever $A\neq B$ for all $A,B\in X$, and $A\neq\varnothing$ for all $A\in X$.
		\end{enumerate}
	\end{defn}
	
	Typically, we will be working with finite sets. In that case, the distinction between Dynkin systems and pre-Dynkin systems is obviously meaningless.
	
	\begin{defn}
		Let $\phi\colon S\to T$ be a function in $\Set$. Then we define:
		\begin{enumerate}
			\item A function $\Dynk(\phi)\colon\Dynk(S)\to\Dynk(T))$ by \[\Dynk(\phi)(X)=\{A\subseteq T:\phi^{-1}(A)\in X\}\]
			\item A function $\Part(\phi)\colon\Part(S)\to\Part(T)$ by setting $\Part(\phi)(X)$ to be the (partition associated to the) smallest equivalence relation on $T$ generated by the image of (the equivalence relation associated to) $X$.
		\end{enumerate}
	\end{defn}

\begin{rmk}\label{rmk:explicitpartitionblocks}
	We can define $\Part(\phi)(X)$ more explicitly. Note that the image of the equivalence relation associated to $P$, under $\phi$, is automatically a symmetric relation. Closing under reflexivity means adding a partition block $\{x\}$ for each $x\in T\setminus\im(\phi)$. Closing under transitivity means placing $x,y\in T$ in the same partition block whenever there is a finite sequence $\{x=x_0,x_1,\ldots,x_m=y\}$ such that for every $0\leq i\leq m$, $x_i$ and $x_{i+1}$ are both in $\phi(P_{k_i})$ for some $k_i$. We may think of this as merging chains of overlapping blocks into single blocks.
\end{rmk}

The next two results follow from standard arguments.

\begin{lem}\label{lem:inducedpartitionsuction}
	Let $P=\{P_i\}$ be a partition of a set $S$ and $\phi\colon S\to T$ a function. If $Q\in\Part(\phi)(P)$ and $\phi(P_i)\cap Q\neq\varnothing$ then $\phi(P_i)\subseteq Q$.
\end{lem}

	\begin{prop}
		The assignments~ $\Dynk$ and ~$\Part$ assemble into functors $\Finp\to\Set_\ast$, where the basepoint of ~$\Dynk(\ulp{n})$ is the powerset Dynkin system on $\ulp{n}$ and the base point of~ $\Part(\ulp{n})$ is the partition in which all blocks contain exactly one element.
	\end{prop}
	
	\begin{rmk}
		It was very helpfully pointed out to me (essentially simultaneously) by Bastiaan Cnossen, Joe Moeller and David W\"arn that there is another description of the functor $\Part$. Recall that the data of a partition of $\ulp{n}$ with $k$ blocks is equivalent to the data of a surjection $\pi\colon\ulp{n}\to\ul{k}$. Therefore we may think of $\Part$ as taking $\ulp{n}$ to the set of surjections with domain $\ulp{n}$ in $\Fin$, the category of finite sets. It is then relatively straightforward to check that if $\ulp{n}\to\ul{k}$ is a surjection and $\phi\colon \ulp{n}\to\ulp{m}$ is a pointed function then the induced partition of $\ulp{m}$ is the one associated to the surjection in the right hand side of the following pushout diagram:
		\[\begin{tikzcd}
			{\ulp{n}} & {\ulp{m}} \\
			{\ul{k}} & {\ul{k}\coprod_{\ulp{n}}\ulp{m}}
			\arrow["\phi", from=1-1, to=1-2]
			\arrow["\pi"', two heads, from=1-1, to=2-1]
			\arrow[two heads, from=1-2, to=2-2]
			\arrow[from=2-1, to=2-2]
			\arrow["\lrcorner"{anchor=center, pos=0.125, rotate=180}, draw=none, from=2-2, to=1-1]
		\end{tikzcd}\]
	W\"arn in particular noticed that this functor has an elegant description as the canonical ``presheaf of subobjects'' on $\Fin^{\op}$.
	\end{rmk}
	
	\subsection{Plasmas and $\fun$-modules}
	
	We now recall the basic definitions from \cite{beardsleynakamura}. Following that, we point out that the plasmas or mosaics (using \cite{beardsleynakamura} or \cite{nakamurareyes-mosaics} respectively) associated to the discrete geometries on $n$ elements are easily recognized as coproducts of the (underlying plasma of the) Krasner hyperfield with itself. 
	\begin{defn}
		We define a \textit{plasma} to be a set $X$ equipped with a hyperoperation $\star\colon X\times X\to \mathcal{P}(X)$ and a ``weak identity'' $0\in X$ such that:
		\begin{enumerate}
			\item For every $x,y\in X$, $x\star y=y\star x$.
			\item For every $x\in X$, $x\in x\star 0$.
		\end{enumerate}
		A morphism of plasmas $(X,\star,0)\to (Y,\boxplus,e)$ is a function $f\colon X\to Y$ such that $f(0)=e$ and $f(x\star x')\subseteq f(x)\boxplus f(x')$ for all $x,x'\in X$. Write $Plas(X,Y)$ for the set of plasma morphisms from $X$ to $Y$.
	\end{defn}

    Plasmas and $\fun$-modules are related by the following adjunction, which is proven in \cite[Theorem 3.13]{beardsleynakamura}.

    \begin{prop}\label{prop:plasModF1adjunction}
There is an adjunction
\[
\begin{tikzcd}[column sep=large]
    \mathrm{Mod}_{\fun}
    \arrow[r, bend left=20, "\tau_{\leq 2}", shift left=1ex, start anchor={[xshift=-.5em,yshift=-.5em]},end anchor={[xshift=.5ex,yshift=-.5ex]}]
    \arrow[r, phantom, "\bot"]
    &
    Plas
    \arrow[l, bend left=20, "\widehat{H}", shift left=1ex, end anchor={[xshift=-1ex,yshift=.5ex]}, start anchor={[xshift=.7ex,yshift=0ex]}]
\end{tikzcd}
\]
in which $\tau_{\leq 2}$ denotes the restriction to the full subcategory of $\Finp$ spanned by $\ulp{0}$, $\ulp{1}$ and $\ulp{2}$ and the functor $\widehat{H}$ is fully faithful.
    \end{prop}

	By Proposition \ref{prop:plasModF1adjunction}, every $\fun$-module yields a plasma by restriction. However, there are a few examples that will be especially useful to us.
	
	\begin{exam}[Krasner Hypergroup]
		Let $\mathbb{K}$ denote the set $\{0,1\}$ with the hyperoperation that has $0$ as weak identity and $1+1=\{0,1\}$. This object is a plasma and in fact an Abelian hypergroup, i.e.~its hyperoperation is suitably associative. It also has a compatible multiplicative structure with multiplicative inverses making it into a so-called hyperfield, but we will not use that structure in this paper.
	\end{exam}
	
	\begin{rmk}
		We can think of $\KK$ as being obtained by taking the cofree plasma on the singleton set $\{1\}$. In other words, the right adjoint of the forgetful functor $Plas\to\Set$ takes the set $\{1\}$ to $\KK$ (note that the \textit{free} plasma on a point is the set $\{0,1\}$ with $1+1=\varnothing\in\mathcal{P}(\KK)$). Therefore, $\KK$ is what we get if take the cofree plasma associated to the projective point $(\mathbb{C}-\{0\})/\mathbb{C}^\ast$.
	\end{rmk}
	
	\begin{exam}[Power Set Plasmas]\label{example:powersetplasma}
		Let $S$ be a set and $\mathcal{P}(S)$ its power set. The power set admits a plasma structure given by setting $a\star b=a\cup b$ whenever $a$ and $b$ are disjoint, and $a\star b=\{~\}$ whenever $a$ and $b$ are non-disjoint, where $\{~\}$ indicates the empty family in $\mathcal{P}(\mathcal{P}(n))$ (not to be confused with $\varnothing\in\mathcal{P}(n)$). Note that, in general, this hyperoperation is not associative. 
	\end{exam}
	
	\begin{rmk}
		The power set plasma of $\ul{n}$ is $\tau_{\leq 2}$ applied to the $\fun$-module corepresented by $\ulp{n}$. Equivalently, $\mathcal{P}(n)$ is the restriction plasma of the free $\fun$-module on $\ulp{n}$ (with respect to the Day convolution monoidal structure).
	\end{rmk}
	
	The power set plasmas, as suggested by their construction as free objects, play a particularly important role in the relationship between plasmas and $\fun$-modules. Indeed, they assemble into a kind of ``nerve'' functor which recovers the ``Eilenberg-MacLane'' functor of Proposition \ref{prop:plasModF1adjunction}. The following two results are Lemma 4.7 and Theorem 4.9 of \cite{beardsleynakamura} respectively.
	
	\begin{prop}\label{prop:PnIscoGammaPlasma}
		The assignment $\ulp{n}\mapsto\mathcal{P}(n)$ with the above plasma structure extends to a functor $\mathcal{P}\colon\Finp^{\op}\to Plas$.
	\end{prop}
	
	\begin{prop}\label{prop:Hiscorepresented}
		If $\yo\colon Plas\to\Fun(Plas^{\op},\Setp)$ denotes the Yoneda embedding and $(\mathcal{P}^{\op})^\ast\colon \Fun(Plas^{\op},\Set_\ast)\to \Fun(\Fin_\ast,\Set_\ast)$ denotes the functor which precomposes with $\mathcal{P}^{\op}\colon \Finp\to Plas^{\op}$, then the embedding $\widehat{H}\colon Plasm\hookrightarrow\mathrm{Mod}_{\fun}$ of Proposition \ref{prop:plasModF1adjunction}  is isomorphic to $(\mathcal{P}^{\op})^\ast\circ\yo$. In particular, $\widehat{H}M(\ulp{n})$ is functorially isomorphic to $Plas(\mathcal{P}(n),M)$. 
	\end{prop}
	
		\begin{exam}[Projective Geometries]\label{exam:projgeoms}
		In \cite{beardsleynakamura}, Nakamura and I extended the results of \cite{nakamurareyes-mosaics} to show that there is a fully faithful embedding of the category of combinatorial projective geometries (in the sense of \cite{faure-frolicher}), with collineations as morphisms, into the category of plasmas. We will write $\mathbb{P}\colon Proj\hookrightarrow Plas$ for this functor.
	\end{exam}  
	
	As a category, plasmas are relatively nice objects, but we will only need the following basic categorical fact in what follows.
	
	\begin{prop}\label{prop:plasmacoprod}
		The category of plasmas admits finite coproducts.
	\end{prop}
	
	\begin{proof}
		It is an exercise in basic category theory to check that if $(P,\star,0)$ and $(Q,\boxplus,e)$ are plasmas then their coproduct in $Plas$ is the wedge sum of pointed sets (taking $0$ and $e$ as base points) $P\vee Q$ with the hyperoperation:\[x\vee y=\begin{cases}
			x\star y &\text{if }x,y\in P\\
			x\boxplus y &\text{if }x,y\in Q\\
			\varnothing & \text{otherwise}
		\end{cases}\]
		The new identity element of $P\vee Q$ is the equivalence class $\{0,e\}$ in the wedge sum.
	\end{proof}
	
	\section{Discrete Geometries are Free $\KK$-Modules}\label{sec:plasmasofgeometries}
	
	\begin{thm}\label{thm:ProjisK}
		Let $\mathbb{P}\{n\}$ denote the discrete combinatorial projective geometry on the set $\ul{n}=\{1,2,\ldots,n\}$. Then the plasma associated to $\mathbb{P}\{n\}$ by the embedding of \cite{nakamurareyes-mosaics,beardsleynakamura} is the $n$-fold coproduct $\vee_{n}\mathbb{K}$.
	\end{thm}
	
	\begin{proof}
		We trace through the constructions of \cite{faure-frolicher,nakamurareyes-mosaics}. The discrete projective geometry on the set $\ul{n}=\{1,\ldots,n\}$ is the one whose ternary collinearity relation $\ell\in\ul{n}\times\ul{n}\times\ul{n}$  (as in \cite[Definition 2.1.1]{faure-frolicher}) is given by $(x,y,z)\in\ell$ exactly when $|\{x,y,z\}|\leq 2$. In other words, given a pair of points $x,y\in\ul{n}$ we have that $x$ and $y$ are on a unique line that contains no other points. By \cite[Proposition 2.2.3]{faure-frolicher}, this is equivalent data to a hyperoperation $\star_{\ell}\colon \ul{n}\times\ul{n}\to\mathcal{P}(\ul{n})$ given by \[ 
		x\star_{\ell}y=\begin{cases}
			\{z\in\ul{n}:(z,x,y)\in\ell\}&\text{if }x\neq y	\\
			\{x\}&\text{if }x=y
		\end{cases}
		\] For $x\neq y\in\ul{n}$ it follows that $x\star_{\ell}y=\{x,y\}$. This hyperoperation is \textit{not} the one which allows projective geometries to be embedded in plasmas. 
		
		From \cite[Proposition 2.3.3]{faure-frolicher} we have a closure operator on $\ul{n}$, denoted $C_\ell\colon\mathcal{P}(\ul{n})\to\mathcal{P}(\ul{n})$ associated to $\ell$ which makes $\ul{n}$ into a simple matroid. In this case, $C_\ell$ is the identity function since every subset of $\ul{n}$ is a ``subspace.'' To use the equivalence of \cite{nakamurareyes-mosaics}, we must pass to simple \textit{pointed} matroids. This is the functor which simply appends an additional point $\{0\}$ and defines a new closure operator $C_\ell^+\colon\mathcal{P}(\ulp{n})\to\mathcal{P}(\ulp{n})$ by setting $C_\ell^+(A)=C_\ell(A)\cup\{0\}$ and $C_\ell^+(0)=\varnothing$ (in other words, we freely adjoin a so-called ``loop'' to the matroid structure of $\ul{n}$). 
		
		Finally, following \cite[Section 4.3]{nakamurareyes-mosaics}, we obtain a plasma (in fact, a slightly less general object that they refer to as a \textit{mosaic}) from $(\ulp{n},C_\ell^+)$ by taking $0$ to be the additive identity and taking \[x\boxplus_\ell y=\begin{cases}
		\varnothing & \text{if }x\neq y\\
						 \{x,0\} & \text{if }x=y
		\end{cases}\]
		In light of Proposition \ref{prop:plasmacoprod}, this is clearly the definition of the coproduct plasma $\vee_{n}\mathbb{K}$.
	\end{proof}
	
	\begin{rmk}\label{rmk:twopresentationsoftheprojectivepoint}
		Perhaps it's worth pointing out that, in the case that $n=1$, this result has some sensible intuitive basis. A common way of constructing hyperrings is to start with a commutative ring $R$ and a subgroup $G\subseteq R^\times$ of the group of units of $R$. The set-theoretic quotient $R/G$ still has a well-defined multiplication given on cosets by $rG\cdot sG=rsG$, but its addition no longer makes sense. Instead, the addition becomes a (associative, commutative, distributive over multiplication) hyperoperation in a very natural way by defining \[rG\boxplus sG=\{u+v:u\in rG,~v\in sG\}/G\] In other words, to add two cosets, we take all possible sums of all possible representatives of each, and then quotient by $G$ again. This is clearly a subset of $R/G$. Applying this construction to any field with more than two elements, say $\mathbb{C}$, and its group of units, results in the Krasner hyperfield. Therefore, we might think of $\KK$ as being the natural algebraic structure of, e.g., the complex projective point $\mathbb{C}/\mathbb{C}^\ast$, where we have not forgotten the existence of an additive identity.
		
		Alternatively, the proof of Theorem \ref{thm:ProjisK} gives a concrete construction of $\KK$ from the one-point geometry. We begin with one point and the trivial collinearity relation. We then take the associated \textit{matroid} and notice that it's simple. Finally, we freely adjoin a ``basepoint'' to this matroid, to get a simple \textit{pointed} matroid. This final construction does not lose any information about the matroid and can be reversed by simply deleting the appended basepoint. The the results of \cite{nakamurareyes-mosaics} tell us that simple pointed matroids map faithfully into plasmas. However, when we restrict to the simple pointed matroids that arise from geometries, this is a fully faithful inclusion. In other words, there is no difference, in terms of information content, between combinatoriaal projective geometries and the class of plasmas associated to them. Therefore this is another way in which we can think of $\KK$ as a projective point, so it perhaps not completely surprising that the two constructions coincide. 
	\end{rmk}

	\begin{rmk}\label{rmk:myKisnotalgebraic}
		In \cite[Proposition 3.1]{connesconsani-hyperringofadeles}, Connes and Consani show that combinatorial projective geometries with at least four points per line can be encoded as hyperring extensions of the Krasner hyperfield. Their construction is generalized in \cite{nakamurareyes-mosaics} to associate mosaics to a much more general class of projective geometries. Theorem \ref{thm:ProjisK} is essentially a computation with this machinery, showing that the discrete geometries on sets with $n$-elements, which one might think of as the ``free'' projective geometries on such sets, correspond to $n$-dimensional vector spaces over $\KK$. It's worth noting, however, that the $\fun$-module associated to $\vee_n\KK$ by the embedding of \cite{beardsleynakamura} is not an $\fun$-\textit{algebra} in any obvious way and probably cannot be thought of as an extension of the hyperfield $\KK$.
	\end{rmk}
	
	We now record a related result that is not immediately relevant but that we will use later.
	
		\begin{prop}\label{prop:P(n)toVKfactorsthroughsummand}
		The inclusion \[\vee_{k}Plasm(\mathcal{P}(n),\KK)\hookrightarrow Plasm(\mathcal{P}(n),\vee_k\KK)\] given by composing with the $k$ different inclusions $\KK\hookrightarrow \vee_k\KK$, is an isomorphism.
	\end{prop}
	
	\begin{proof}
		We prove the case that $k=2$, as the other cases are essentially the same but involve more indices. Suppose $f\colon \mathcal{P}(n)\to\KK\vee\KK=\{0,1,2\}$ is a plasma morphism that does not factor through $\{0,1\}$ or $\{0,2\}$. Then there are some $A\subseteq\ul{n}$ and $B\subseteq \ul{n}$ such that $f(A)=1$ and $f(B)=2$. Suppose that $A\cap B=\varnothing$. Then $f(A\cup B)\subseteq 1+2=\varnothing$. But $A\cup B\in\mathcal{P}(n)$ so $f(A\cup B)\in \KK\vee \KK$. Thus it cannot be that $A\cap B=\varnothing$, and of course it cannot be true that $A=B$ either.
		Now we note the following:
		\begin{align*}
			2=f(B)&=f((B\setminus A)\cup (B\cap A))\\
			&\subseteq f(B\setminus A)+f(B\cap A)
		\end{align*}
		Therefore one of $f(B\setminus A)$ and $f(B\cap A)$ must be $2$. Suppose it is $f(B\cap A)$. Then
		\begin{align*}
			1=f(A)&=f(A\setminus B\cup (A\cap B))\\
			&\subseteq f(A\setminus B)+f(A\cap B)\\
			&=f(A\setminus B)+2
		\end{align*}
		 But this cannot be the case, as $x+2$ is either $\{2\}$, $\{0,2\}$ or empty for any $x\in \KK\vee\KK$. Therefore it must be the case that $f(B\cap A)=0$, $f(A\setminus B)=1$ and $f(B\setminus A)=2$. But then $f((A\setminus B)\cup (B\setminus A))\subseteq 1+2=\varnothing$, which is a contradiction. Therefore $f$ must factor through either the left or right coordinate.
	\end{proof}

	\section{Dynkin Systems as Kernels}\label{sec:Dynkinsaskernels}
	
	In this section we show that the data of what we call a KZ-system on $\{1,2,\ldots,n\}$, which we will see is always the zero set of a plasma morphism $\mathcal{P}(n)\to\mathbb{K}$, is equivalent data to a Dynkin system on $\{0,1,2,\ldots,n\}$. For an infinite set $S$, KZ-systems are equivalent to pre-Dynkin systems. It will follow that (pre-)Dynkin systems can be identified with plasma morphisms to $\KK$ by taking kernels.
	
	\begin{defn}
		Let $S$ be a set and $X\subseteq\mathcal{P}(S)$. Then we say that $X$ is a KZ-system if:
		\begin{enumerate}
			\item $\varnothing\in X$.
			\item For any $A,B\in X$, if $A\cap B=\varnothing$ then $A\cup B\in X$.
			\item If $A\in X$ and $B\notin X$ and $A\cap B=\varnothing$ then $A\cup B\notin X$.
		\end{enumerate}
		Given a set $S$, write $KZ(S)$ for the set of $KZ$-systems on $S$.
	\end{defn}
	
	\begin{prop}
		There is a bijection $\kappa\colon Plas(\mathcal{P}(n),\mathbb{K})\to KZ(\{1,\ldots, n\})$ given by $\kappa(f)=f^{-1}(0)$.
	\end{prop}
	
	\begin{proof}
		Let $f\colon \mathcal{P}(n)\to\mathbb{K}$ be a function so that $f^{-1}(0)$ is a family of subsets of $\{1,2,\ldots, n\}$. Clearly $\varnothing\in f^{-1}(0)$ if and only if $f(\varnothing)=0$. Let $A,B\in\mathcal{P}(n)$ with $A\cap B=\varnothing$. Because $0+0=\{0\}$ in $\mathbb{K}$ and $0+1=1+0=\{1\}$, $f$ satisfies the remaining conditions of being a plasma morphism if and only if $f^{-1}(0)$ is a KZ-system. Specifically, if $f(A)=f(B)=0$ then $A\cup B\in f^{-1}(0)$ exactly when $f(A\cup B)\subseteq f(A)+f(B)$. If $f(A)=1$ and $f(B)=0$ then $f(A\cup B)=1$ is equivalent to requiring $f(A\cup B)\subset f(A)+f(B)=\{1\}$. The case that $f(A)=f(B)=1$ induces no restrictions at all.
	\end{proof}
	
	\begin{thm}\label{thm:KZ to prelambda}
		Let $S$ be a set. Then there is a bijection between $KZ$-systems on $S$ and pre-Dynkin-systems on $S_+=S\cup\{0\}$ \[\Phi\colon KZ(S)\to \Dynk^{pre}(S_+)\]
		given by the formula \[\Phi(X)=X\cup X^\perp_+\] where $X^\perp_+=\{S_+\setminus A:A\in X\}$. The inverse function is given by $\Psi(Y)=\{A\in Y:0\notin A\}$.
	\end{thm}
	
	\begin{proof}
		We first show that $\Phi(X)$ is a Dynkin system. By construction, $\Phi(X)$ is closed under complements and contains $\varnothing$. It remains to show that $\Phi(X)$ is closed under taking unions of pairs of disjoint sets. Let $A,B\in\Phi(X)$ and assume $A\cap B=\varnothing$. If $A,B\in X$ then, by assumption, $A\cup B\in X\subseteq \Phi(X)$. Moreover, it cannot be the case that $A$ and $B$ are both elements of $\Phi(X)-X$ since that would imply they both must contain $0$ and therefore be non-disjoint. Thus it suffices to prove the statement in the case that $A\in X$ and $B\in X_+^\perp$. 
		
		The fact that $B\in X^\perp_+$ implies that $B=(S\setminus D)\cup\{0\}$ for some $D\in X$. The assumption that $A\cap B=\varnothing$ implies that $A\cap(S\setminus D)=\varnothing$ and therefore that $A\subseteq D$. Hence it must be the case that $A\cup ((S\setminus A)\cap D)=D\in X$. But since $A\in X$ and $X$ is a KZ-system we must have that $(S\setminus A)\cap D$ is also an element of $X$ (since if it were not then $D$ would also not be in $X$, by the second axiom). Therefore, by construction, $S\setminus ((S\setminus A)\cap D)\cup\{0\}\in X_+^\perp\subseteq\Phi(X)$. But this last set is equal to $A\cup (S\setminus D)\cup\{0\}=A\cup B$. We have shown that $\Phi(X)$ is a pre-Dynkin system.
		
		We now check that $\Phi$ is injective and surjective. For injectivity, suppose that $\Phi(X)=\Phi(Y)$ for $X,Y\in KZ(S)$. Then $X=\{A\in\Phi(X):0\notin A\}=\{A\in\Phi(Y):0\notin A\}=Y$. For surjectivity, let $Y\in \Dynk^{pre}(S_+)$ and define $\Psi(Y)=\{A\in Y:0\notin Y\}$. We first show that $\Psi(Y)$ is a KZ-system.
		
		It is clear that $\Psi(Y)$ contains $\varnothing$ and that it is closed under union of disjoint elements. It remains to show that if two sets are disjoint and one of them is \textit{not} in $\Psi(Y)$ then their union is not in $\Psi(Y)$. Suppose that $A,B\subseteq S$, $A\in \Psi(Y)$, $B\notin\Psi(Y)$ and $A\cap B=\varnothing$. Suppose that $A\cup B\in\Psi(Y)$. Then since $\Psi(Y)\subseteq Y$, it must be that $A\cup B \in Y$. Since $Y$ is closed under complements, $((S\setminus A)\cap (S\setminus B))\cup\{0\}$ is also in $Y$. This last set is disjoint from $A$, so $A\cup ((S\setminus A)\cap (S\setminus B))\cup\{0\}=(S\setminus B)\cup\{0\}$ is in $Y$. It follows that $B=S_+\setminus((S\setminus B)\cup\{0\})\in Y$, again by closure under complements, and therefore $B\in \Psi(Y)$. But this is a contradiction, so it cannot be the case that $A\cup B\in \Psi(Y)$. Thus $\Psi(Y)$ is a KZ-system on $S$. It is not hard to check that $\Phi\Psi(X)=X$.
	\end{proof}
	
	\begin{cor}\label{cor:dynkinsystemsarekernels}
		For any set $S$ there is a bijection $Plas(\mathcal{P}(S),\mathbb{K})\to \Dynk^{pre}(S_+)$ given by $f\mapsto f^{-1}(0)\cup (f^{-1}(0)^\perp_+)$. If $S$ is finite this clearly induces a bijection $Plas(\mathcal{P}(S),\mathbb{K})\cong \Dynk(S_+)$.
	\end{cor}
	
	\begin{rmk}
			One place (pre-)Dynkin systems appear is in probability theory, especially, it seems, in so-called quantum probability theory \cite{suppes-probabilisticquantum,gudder-quantumprobabilityspaces}. I know little to nothing about probability theory, quantum or otherwise, but I wonder if there is a probabilistic interpretation of these functions $\mathcal{P}(S)\to\KK$. Na\"ively, one might note that we are associating either a $0$ or $1$ to each subset of $S$ which, in certain contexts, can be thought of as a kind of probability measure, specifically a ``two-valued probability measure,'' though it is \textit{subadditive}, in a certain sense, rather than additive. Indeed, its defining formulas are not totally unlike the definition of an ``upper probability measure'' in, for instance, Section 2 of \cite{walley-frequentist}.
	\end{rmk}
	
	Using Proposition \ref{prop:Hiscorepresented}, we now know that the functor \[Plas(\mathcal{P}(-\setminus \{0\}),\KK)\cong\widehat{H}\KK\colon \Finp\to\Setp\] agrees with $\Dynk$ on objects, but it remains to check that they are isomorphic as functors. It suffices to show naturality of the above isomorphisms.
	
	\begin{prop}\label{prop:naturality of Phi}
		For any pointed function $\phi\colon\ulp{n}\to\ulp{m}$, the following square commutes\[
		\begin{tikzcd}
			Plas(\mathcal{P}(n),\KK)\ar[d,"{Plas(\mathcal{P}(\phi),\mathbb{K})}",swap] \ar[r, "\Phi_{\ulp{n}}"] & \Dynk(\ulp{n})\ar[d,"\Dynk(\phi)"]\\
			Plas(\mathcal{P}(m),\KK)\ar[r,"\Phi_{\ulp{m}}",swap] & \Dynk(\ulp{m})
		\end{tikzcd}\]
		with $Plas(\mathcal{P}(\phi),\KK)$ denoting the function that precomposes a plasma morphism with $\phi^{-1}\colon\mathcal{P}(m)\to\mathcal{P}(n)$.
	\end{prop}
	\begin{proof}
		Let $f\colon \mathcal{P}(n)\to\KK$ be a morphism of plasmas. The function $Plas(\mathcal{P}(\phi),\KK)$ applied to $f$ gives the function $(\phi^{-1})_\ast(f)\colon \mathcal{P}(m)\to\KK$ defined by $U\mapsto f(\phi^{-1}(U))$. The associated element of $\Dynk(\ulp{m})$ comprises subsets $U\subseteq\ulp{m}\setminus\{0\}$ with $f(\phi^{-1}(U))=0$  along with their complements $W=\ulp{m}\setminus U$ (which necessarily contain $0$). These last are precisely the subsets $W\subseteq\ulp{m}$ such that \[f(\phi^{-1}(\ulp{m}\setminus W))=f(\ulp{n}\setminus\phi^{-1}(W))=f(\ul{n}\setminus(W\setminus\{0\}))=0\]
		
		On the other hand $\Phi_{\ulp{n}}(f)$, the element of $\Dynk(\ulp{n})$ associated to $f$, is composed of each set $V\subseteq\ulp{n}$ with $f(V)=0$ and its complement $\ulp{n}\setminus V$. The Dynkin system associated to this by $\Dynk(\phi)$ is the collection of subsets $W\subseteq\ulp{m}$ such that $\phi^{-1}(W)\in\Phi_{\ulp{n}}(f)$. These are the sets $W$ such that either $f(\phi^{-1}(W))=0$ in the case that $0\notin W$, or \[f(\phi^{-1}(\ulp{m}\setminus W))=f(\ulp{n}\setminus \phi^{-1}(W))=f(\ul{n}\setminus (\phi^{-1}(W\setminus\{0\})))=0\] in the case that $0\in W$.		
	\end{proof}

	Any $\fun$-module $X$ has a delooping $BX\colon\Delta^{\op}\to\Set_\ast$ as in \cite[Appendix A]{beardsleynakamura}. In the case that $X$ is $\widehat{H}A$ for $A$ an Abelian group, this is precisely the usual bar construction whose geometric realization is the classifying space of $A$. In this sense, the functor $B\widehat{H}$ properly generalizes the usual classifying space functor, or Segal's ``delooping'' functor \cite{segal}. In this language, the preceding results lead to the following results.
	
	\begin{thm}\label{thm:F1modofprojisDynk}
		The $\fun$-module of the discrete projective geometry on $n$-points is isomorphic to the $n$-fold coproduct of the Dynkin $\fun$-module with itself. 
	\end{thm}
	
	\begin{proof}
		This follows immediately from combining Propositions \ref{prop:Hiscorepresented}, \ref{prop:naturality of Phi} and \ref{prop:P(n)toVKfactorsthroughsummand}.
	\end{proof}
	
	\begin{cor}\label{cor:deloopingofKisDynk}
	The delooping of the Krasner hyperfield~$\KK$ is isomorphic to the delooping of the Dynkin $\fun$-module.
	\end{cor}

	\section{The Sub-$\fun$-module of Partitions}\label{sec:F1modofPartitions}
	
	Recall that by Dynkin's well-known $\pi$-$\lambda$-theorem (see \cite[A.1.4]{durrett-probability} or \cite[Lemma 1.1]{dynkinTheory}), a Dynkin system which is closed under finite intersections is a $\sigma$-algebra. Moreover every $\sigma$-algebra on a finite set is an atomic Boolean algebra, and the set of atoms of that algebra forms a partition of the finite set. Indeed, this induces a bijection between $\sigma$-algebras on finite sets and partitions of finite sets. In this section we see that the $\fun$-module of partitions is a proper submodule of that of Dynkin systems described above.
	
	We know immediately that there are proper subset inclusions $\Part(\ulp{n})\subseteq\Dynk(\ulp{n})$, so it only remains to check that the functions $\Dynk(\phi)\colon \Dynk(\ulp{n})\to\Dynk(\ulp{m})$ preserve the property of being intersection closed and, after passing isomorphically to partitions, recover the functions $\Part(\phi)$.
	
	\begin{lem}
		Suppose that $X\in\Dynk(\ulp{n})$ is closed under intersections, and therefore determines a $\sigma$-algebra, or equivalently a partition, on $\ulp{n}$. If $\phi\colon \ulp{n}\to\ulp{m}$ is a function of pointed sets then the Dynkin system $\Dynk(\phi)(X)\in\Dynk(\ulp{m})$ is also closed under intersections.
	\end{lem}
	
	\begin{proof}
		Let $A,B\in\Dynk(\phi)(X)$, i.e.~they are subsets of $\ulp{m}$ such that $\phi^{-1}(A)$ and $\phi^{-1}(B)$ are elements of $X$. Note that $\phi^{-1}(A\cap B)=\phi^{-1}(A)\cap\phi^{-1}(B)$ and therefore $\Dynk(\phi)(X)$ is closed under intersection whenever $X$ is.
	\end{proof}

	\begin{prop}\label{prop:partitionsubfunctor}
		For every $\phi\colon\ulp{n}\to\ulp{m}$ in $\Finp$, the following diagram commutes:
		\[\begin{tikzcd}
			{\Part(\ulp{n})} & {\Part(\ulp{m})} \\
			{\Dynk(\ulp{n})} & {\Dynk(\ulp{m})}
			\arrow["{\Part(\phi)}", from=1-1, to=1-2]
			\arrow[hook, from=1-1, to=2-1, "\Sigma", swap]
			\arrow[hook, from=1-2, to=2-2, "\Sigma"]
			\arrow["{\Dynk(\phi)}"', from=2-1, to=2-2]
		\end{tikzcd}\]
		in which the vertical arrows are given by freely generating a $\sigma$-algebra from the partition blocks (under intersection and union).
		
	\end{prop}
		
	\begin{proof}
		
		Let $P=\{P_i\}_{i\in I}$ be a partition of $\ulp{n}$. We need to show that $X\in \Sigma(\Part(\phi)(P))$ if and only if $X\in \Dynk(\phi)(\Sigma(P))$. Because every function factors as the composite of an injection and a surjection, it suffices to separately consider the case that $\phi$ is an injection and the case that $\phi$ is a surjection.
		
		Suppose that $\phi$ is an injection and note that, in this case, $\{\phi(P_i)\}_{i\in I}$ forms a partition of $\im(\phi)$. Appending the family $\{\{x\}:x\notin\im(\phi)\}$ gives the partition $\Part(\phi)(P)$ of $\ulp{m}$ (in other words, we need only take reflexive closure, following Remark \ref{rmk:explicitpartitionblocks}). Now let $X\in\Sigma(\Part(\phi)(P))$, i.e.~let $X$ be a union of partition blocks of $\Part(\phi)(P)$. Using this description of the blocks of $\Part(\phi)(P)$, we have that the preimage of $X$ is the union of a collection of $P_i$, the blocks of $P$. Therefore, by definition, $\phi^{-1}(X)\in\Sigma(P)$ and $X\in\Dynk(\phi)(\Sigma(P))$.
		
		Still supposing that $\phi$ is an injection, let $X\in\Dynk(\phi)(\Sigma(P))$, i.e.~$X\subseteq\ulp{m}$ and $\phi^{-1}(X)=\cup_{j\in J}P_j$ for some $J\subseteq I$. We are trying to assemble $X$ from sets of the form $\phi(P_i)$ and $\{x\}$ for $i\in I$ and $x\notin\im(\phi)$. Since we are free to add any subset of $\ulp{m}\setminus\im(\phi)$ (i.e.~any collection of singletons $\{x\}$ for $x\notin\im(\phi)$ can be added or removed without changing the preimage of $X$) we may assume that $X$ is contained in the image of $\phi$. Then we may write $X$ as $\cup_{j\in J}\phi(P_j)\in\Sigma(\Part(\phi)(P))$.
		
		Now suppose that $\phi$ is a surjection and let $X\in\Sigma(\Part(\phi)(P))$. Therefore $X$ can be written uniquely as a union of the disjoint blocks of $\Part(\phi)(P)$, each of which is a union of (not necessarily disjoint) sets of the form $\phi(P_i)$ (again, refer to Remark \ref{rmk:explicitpartitionblocks} for a more explicit description of this partition). Let us write $X=Q_1\cup Q_2\cup\cdots\cup Q_k$ for the partition of $X$ into blocks of $\Part(\phi)(P)$. For $Q_i$, write $Q_i=\phi(P_{i_1})\cup\cdots\cup \phi(P_{i_{d_i}})$ for the unique way of writing $Q_i$ as a union of images of blocks of $P$. Now let \[K=\{1_1,1_2,\ldots,1_{d_1},2_1,\ldots,2_{d_2},\ldots,\ldots,k_{d_k}\}\subseteq I\] be the set of indices of $I$ used in constructing $X$. We show that $\phi^{-1}(X)=\cup_{i\in K}P_i$.
		
		Suppose $x\in P_j\cap\phi^{-1}(X)$ for some $j\notin K$. Using surjectivity of $\phi$, we have $\phi(x)\in \phi(P_j)\cap X$. Then $\phi(x)\in\phi(P_j)\cap Q_\ell$ for some $1\leq\ell\leq k$. By Lemma \ref{lem:inducedpartitionsuction} it follows that $\phi(P_j)\subseteq Q_\ell\subseteq X$. Hence $P_j\subseteq\phi^{-1}\phi(P_j)\subseteq\phi^{-1}(X)$. This is a contradiction. So $j\in K$ and $\phi^{-1}(X)=\cup_{i\in K}P_i$. It follows that $X\in\Dynk(\phi)(\Sigma(P))$.
		
		Continuing to assume $\phi$ is a surjection, suppose that $X\in\Dynk(\phi)(\Sigma(P))$. Then $\phi^{-1}(X)\in\Sigma(P)$ so $\phi^{-1}(X)=\cup_{i\in K}P_i$ for some $K\subseteq I$. By surjectivity, it follows that $X=\cup_{i\in K}\phi(P_i)$. We now check that $X$ is closed under transitivity so that it is a union of blocks of $\Part(\phi)(P)$. Let $x\in X$ so that $\phi^{-1}(x)\subseteq\phi^{-1}(X)=\cup_{i\in K}P_i$. Let $\phi(y)=x$ for some $y\in P_\ell$ with $\ell\notin K$. Then $y\in\phi^{-1}(X)=\cup_{i\in K}P_i$. Then $P_\ell\cap P_i\neq\varnothing$ for some $i\in K$ and $\ell\notin K$. But the $P$ is a partition, so it must be the case that $i=\ell$, a contradiction.		
	\end{proof}

    \begin{rmk}
        By using the delooping functor of \cite[Appendix A]{beardsleynakamura} one can compute the simplicial set associated to the $\fun$-module $\Part$. I encourage the reader to ``draw some pictures'' to see the way in which face maps and degeneracy maps modify partitions. A face map $\Part(\ulp{n})\to\Part(\ulp{n-1})$ that ``adds'' $i$ and $j$ will merge the blocks containing $i$ and $j$, delete $j$, and decrement each index greater than $j$. On the other hand, a degeneracy map $\Part(\ulp{n})\to\Part(\ulp{n+1})$ which ``inserts the identity" in position $i$ corresponds by incrementing each index greater than $i$ and creating a new singleton block $\{i\}$. Note however that blocks can never be deleted, they can only be merged with the block containing the basepoint. I would be interested to know whether or not these sorts of functions have a relationship to Markov processes and coalescents, especially those in which some fixed element is ``followed'' throughout the process, as these are all pointed sets. For instance, all elements sharing the pointed block might be thought of as ``extinct'' in the process.
    \end{rmk}
	
	Proposition \ref{prop:partitionsubfunctor} implies the following.
	
	\begin{thm}\label{thm:partsubdynk}
		The $\fun$-module of partitions of finite sets is a proper submodule of the $\fun$-module of Dynkin systems on finite sets.
	\end{thm}

    One can check that the inclusion $\Part\hookrightarrow\Dynk$ truncates to an isomorphism of plasmas \[\tau_{\leq 2}\Part\xrightarrow{\cong}\tau_{\leq 2}\Dynk\] and that these are both isomorphic to the Krasner hypergroup $\KK$. Since $\widehat{H}$ is fully faithful, we also know that the unit of the $\tau_{\leq 2}\dashv\widehat{H}$ adjunction \[\Dynk\to\widehat{H}\tau_{\leq 2}\Dynk\] is an isomorphism. We can conclude therefore that $\Part$ cannot be the $\fun$-module associated to any plasma. 
	 
	 In this way, $\Part$ is an example of an $\fun$-module which is not, at its core, a set with an operation (even partially defined or multivalued). If we force ourselves to consider it as such, it becomes indiscernible from $\KK$. The truncated binary (hyper)operation of $\Part$ says that the set with elements $a=\{\{0\},\{1\}\}$ and $b=\{\{0,1\}\}$ has $a$ as an ``additive identity,'' and has $b+b=\{a,b\}$. This is a hyperoperation on the set of partitions of $\{0,1\}$. This is the same as $\KK$, but diverges for $\ulp{n}$ with $n>2$. For instance, there are 15 partitions of $\{0,1,2,3\}$ but 19 Dynkin systems. The missing elements in $\Part(\ulp{3})$ can be thought of disallowed ``associators'' in the algebraic structure of partitions inherent in thinking of $\Part$ as an $\fun$-module.

	 \section{The Geometry of the Simplicial Set of Dynkin Systems}\label{sec:Dynkin Systems as Simplices}

	 We now compute some of the low dimensional simplices of the pointed simplicial set \[B\widehat{H}\KK=B\Dynk\colon\Delta^{\op}\to\Setp\] Recall that if $M$ is a commutative monoid then $B\widehat{H}M$ is the bar construction on $M$. It has one 0-simplex and for each element of $M$ it has one $1$-simplex, i.e.~a ``directed path'' (indeed, $B\widehat{H}M$ is always a quasicategory). The composition operation of $B\widehat{H}M$ is precisely the addition of $M$ and the relationships between the elements, of the form $a+b=c$, are encoded by the 2-simplices which have $a$, $b$ and $c$ as faces. Associativity is then enforced by 3-simplices. The delooping of $\KK$ has a similar structure except that the addition is multivalued. As a result there are two 2-simplices whose zeroth and second face (indexed in the standard way by the vertex they are opposed to) are both 1. In other words, its 2-simplices are a (non-disjoint) union of the 2-simplices of the deloopings of $\mathbb{Z}/2$ and the Boolean monoid $\mathbb{B}$. 
	 
	  In dimension 0 of $B\widehat{H}\KK$, we have the unique 0-simplex $\{0\}$ and in dimension 1 we get $\KK=\{0,1\}$, i.e.~two 1-simplices, one of which is degenerate. Thus far, we have exactly the simplicial circle $\Delta^1/\partial\Delta^1$. In dimension 2 we record all possible ways of ``composing'' 1-simplices, or adding elements of $\KK$ (or traversing paths, from a geometric perspective). Geometrically, these give the following five 2-simplices:
	
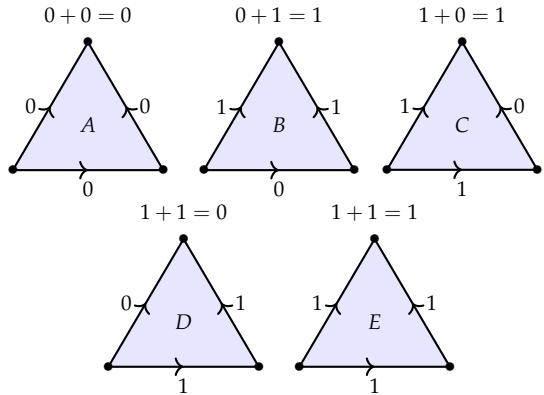
\begin{center}

\begin{tikzpicture}[scale=1, font=\footnotesize, every node/.style={inner sep=1pt}]

	\coordinate (v0) at (0,0);
	\coordinate  (v1) at (2,0);
	\coordinate (v2) at (1,1.7);
	
	\node at (v2) [above=2mm of v2] {$0+0=0$};

	\fill[blue!10] (v0) -- (v1) -- (v2) -- cycle;

	\draw[thick,->-] (v0) -- (v1) node[midway, below=3pt] {$0$};
	\draw[thick,->-] (v1) -- (v2) node[midway, right=4pt] {$0$};
	\draw[thick,->-] (v0) -- (v2) node[midway, left=4pt] {$0$};

	\node  at (v0) {$\bullet$};
	\node  at (v1) {$\bullet$};
	\node  at (v2) {$\bullet$};

	\node at (1,0.6) {$A$};
\end{tikzpicture}
	 ~
	 \begin{tikzpicture}[scale=1, font=\footnotesize, every node/.style={inner sep=1pt}]

	 	\coordinate (v0) at (0,0);
	 	\coordinate (v1) at (2,0);
	 	\coordinate (v2) at (1,1.7);
	 	
	 	\node at (v2) [above=2mm of v2] {$0+1=1$};

	 	\fill[blue!10] (v0) -- (v1) -- (v2) -- cycle;

	 	\draw[thick,->-] (v0) -- (v1) node[midway, below=3pt] {$0$};
	 	\draw[thick,->-] (v1) -- (v2) node[midway, right=4pt] {$1$};
	 	\draw[thick,->-] (v0) -- (v2) node[midway, left=4pt] {$1$};
	 	
	 	\node  at (v0) {$\bullet$};
	 	\node  at (v1) {$\bullet$};
	 	\node  at (v2) {$\bullet$};

	 	\node at (1,0.6) {$B$};
	 \end{tikzpicture}
	 ~\begin{tikzpicture}[scale=1, font=\footnotesize, every node/.style={inner sep=1pt}]
	 	\coordinate (v0) at (0,0);
	 	\coordinate (v1) at (2,0);
	 	\coordinate (v2) at (1,1.7);
	 	
	 		\node at (v2) [above=2mm of v2] {$1+0=1$};

	 	\fill[blue!10] (v0) -- (v1) -- (v2) -- cycle;
	 	
	 	\draw[thick,->-] (v0) -- (v1) node[midway, below=3pt] {$1$};
	 	\draw[thick,->-] (v1) -- (v2) node[midway, right=4pt] {$0$};
	 	\draw[thick,->-] (v0) -- (v2) node[midway, left=4pt] {$1$};

	 	\node  at (v0) {$\bullet$};
	 	\node  at (v1) {$\bullet$};
	 	\node  at (v2) {$\bullet$};
	 	
	 	\node at (1,0.6) {$C$};
	 \end{tikzpicture}
	 ~\\\begin{tikzpicture}[scale=1, font=\footnotesize, every node/.style={inner sep=1pt}]

	 	\coordinate (v0) at (0,0);
	 	\coordinate (v1) at (2,0);
	 	\coordinate (v2) at (1,1.7);
	 	
	 	\node at (v2) [above=2mm of v2] {$1+1=0$};

	 	\fill[blue!10] (v0) -- (v1) -- (v2) -- cycle;

	 	\draw[thick,->-] (v0) -- (v1) node[midway, below=3pt] {$1$};
	 	\draw[thick,->-] (v1) -- (v2) node[midway, right=4pt] {$1$};
	 	\draw[thick,->-] (v0) -- (v2) node[midway, left=4pt] {$0$};

	 	\node  at (v0) {$\bullet$};
	 	\node  at (v1) {$\bullet$};
	 	\node  at (v2) {$\bullet$};
	 	
	 	\node at (1,0.6) {$D$};
	 \end{tikzpicture}
	 ~
	 \begin{tikzpicture}[scale=1, font=\footnotesize, every node/.style={inner sep=1pt}]

	 	\coordinate (v0) at (0,0);
	 	\coordinate (v1) at (2,0);
	 	\coordinate (v2) at (1,1.7);
	 	
	 	\node at (v2) [above=2mm of v2] {$1+1=1$};
	 	
	 	\fill[blue!10] (v0) -- (v1) -- (v2) -- cycle;

	 	\draw[thick,->-] (v0) -- (v1) node[midway, below=3pt] {$1$};
	 	\draw[thick,->-] (v1) -- (v2) node[midway, right=4pt] {$1$};
	 	\draw[thick,->-] (v0) -- (v2) node[midway, left=4pt] {$1$};

	 	\node at (v0) {$\bullet$};
	 	\node at (v1) {$\bullet$};
	 	\node  at (v2) {$\bullet$};
	 	
	 	\node at (1,0.6) {$E$};
	 \end{tikzpicture}
	 \end{center}
	 	 In other words, if we traverse 1 twice this is deformable, in a sense, both to traversing 1 only once and to not moving at all.
	 
	  In dimension 3, elements of $B\widehat{H}\KK(\ulp{3})$ can be interpreted as recording all possible triple traversals of 1-simplices and the multiple ways in which these can thought of as associative. The reason that there are 19 elements is that, due to the hyperoperation of $\KK$ being multivalued, there are several degrees of indeterminacy in deciding how to do this. Here are eight of the 19 elements, with faces labeled as above.
	  
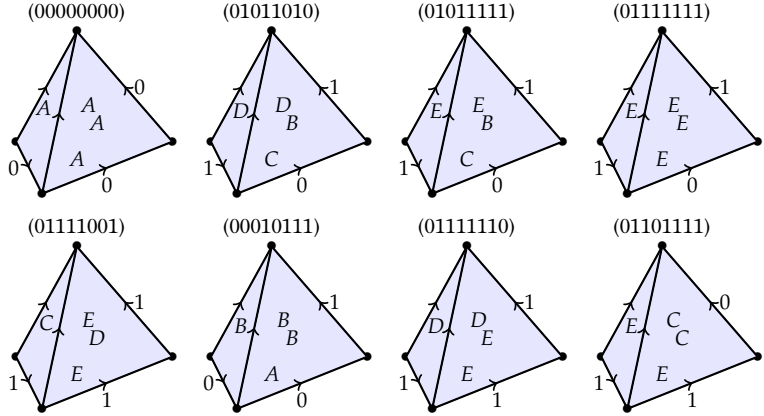
\begin{center}
	  \begin{tikzpicture}[line join = round, line cap = round,scale=1.2,font=\footnotesize]
	  	
	  	\coordinate (3) at (0,{sqrt(2)},0);
	  	\coordinate (2) at ({-.5*sqrt(3)},0,-.5);
	  	\coordinate (1) at (0,0,1);
	  	\coordinate (0) at ({.5*sqrt(3)},0,-.5);
	  	
	  	\node at (3) {$\bullet$};
	  	\node at (2) {$\bullet$};
	  	
	  	\node[above] at (3) {(00000000)};
	  	
	  	\node at (0) {$\bullet$};
	  	
	  	\begin{scope}[decoration={markings,mark=at position 0.5 with {\arrow{to}}}]
	  		\draw[thick, postaction={decorate}] (2)--(0);
	  		\draw[fill=blue!10,fill opacity=.5] (1)--(0)--(3)--cycle;
	  		\draw[fill=blue!10,fill opacity=.5] (2)--(1)--(3)--cycle;
	  		\draw[thick, postaction={decorate}] (1)--(0) node[midway,below]{$0$};
	  		\draw[thick, postaction={decorate}] (2)--(1) node[midway,left]{$0$};
	  		\draw[thick, postaction={decorate}] (2)--(3);
	  		\draw[thick, postaction={decorate}] (1)--(3);
	  		\draw[thick, postaction={decorate}] (0)--(3) node[midway,right]{$0$};
	  	\end{scope}
	  	
	  	\node at (1) {$\bullet$};
	  	\node[opacity=.5] at (barycentric cs:0=0.5,1=0.5 ,2=0.5) {$A$};
	  	\node at (barycentric cs:0=0.5,1=0.5 ,3=0.5) {$A$};
	  	\node at (barycentric cs:1=0.2,2=0.6 ,3=0.5) {$A$};
	  	\node[opacity=.5] at (barycentric cs:0=0.5,2=0.5 ,3=0.5) {$A$};
	  \end{tikzpicture}
	  \begin{tikzpicture}[line join = round, line cap = round,scale=1.2,font=\footnotesize]
	  	
	  	\coordinate (3) at (0,{sqrt(2)},0);
	  	\coordinate (2) at ({-.5*sqrt(3)},0,-.5);
	  	\coordinate (1) at (0,0,1);
	  	\coordinate (0) at ({.5*sqrt(3)},0,-.5);
	  	
	  	\node at (3) {$\bullet$};
	  	\node at (2) {$\bullet$};
	  	
	  	\node at (0) {$\bullet$};
	  	
	  	\node[above] at (3) {(01011010)};
	  	
	  	\begin{scope}[decoration={markings,mark=at position 0.5 with {\arrow{to}}}]
	  		\draw[thick, postaction={decorate}] (2)--(0);
	  		\draw[fill=blue!10,fill opacity=.5] (1)--(0)--(3)--cycle;
	  		\draw[fill=blue!10,fill opacity=.5] (2)--(1)--(3)--cycle;
	  		\draw[thick, postaction={decorate}] (1)--(0) node[midway,below]{$0$};
	  		\draw[thick, postaction={decorate}] (2)--(1) node[midway,left]{$1$};
	  		\draw[thick, postaction={decorate}] (2)--(3);
	  		\draw[thick, postaction={decorate}] (1)--(3);
	  		\draw[thick, postaction={decorate}] (0)--(3) node[midway,right]{$1$};
	  	\end{scope}
	  	
	  	\node at (1) {$\bullet$};
	  	\node[opacity=.5] at (barycentric cs:0=0.5,1=0.5 ,2=0.5) {$C$};
	  	\node at (barycentric cs:0=0.5,1=0.5 ,3=0.5) {$B$};
	  	\node at (barycentric cs:1=0.2,2=0.5 ,3=0.4) {~$D$};
	  	\node[opacity=.5] at (barycentric cs:0=0.5,2=0.5 ,3=0.5) {$D$};
	  \end{tikzpicture}
	  	  \begin{tikzpicture}[line join = round, line cap = round,scale=1.2, font=\footnotesize]
	  	
	  	\coordinate (3) at (0,{sqrt(2)},0);
	  	\coordinate (2) at ({-.5*sqrt(3)},0,-.5);
	  	\coordinate (1) at (0,0,1);
	  	\coordinate (0) at ({.5*sqrt(3)},0,-.5);
	  	
	  	\node at (3) {$\bullet$};
	  	\node at (2) {$\bullet$};
	  	
	  	\node at (0) {$\bullet$};
	  	
	  	\node[above] at (3) {(01011111)};
	  	
	  	\begin{scope}[decoration={markings,mark=at position 0.5 with {\arrow{to}}}]
	  		\draw[thick, postaction={decorate}] (2)--(0);
	  		\draw[fill=blue!10,fill opacity=.5] (1)--(0)--(3)--cycle;
	  		\draw[fill=blue!10,fill opacity=.5] (2)--(1)--(3)--cycle;
	  		\draw[thick, postaction={decorate}] (1)--(0) node[midway,below]{$0$};
	  		\draw[thick, postaction={decorate}] (2)--(1) node[midway,left]{$1$};
	  		\draw[thick, postaction={decorate}] (2)--(3);
	  		\draw[thick, postaction={decorate}] (1)--(3);
	  		\draw[thick, postaction={decorate}] (0)--(3) node[midway,right]{$1$};
	  	\end{scope}
	  	
	  	\node at (1) {$\bullet$};
	  	\node[opacity=.5] at (barycentric cs:0=0.5,1=0.5 ,2=0.5) {$C$};
	  	\node at (barycentric cs:0=0.5,1=0.5 ,3=0.5) {$B$};
	  	\node at (barycentric cs:1=0.2,2=0.5 ,3=0.4) {~$E$};
	  	\node[opacity=.5] at (barycentric cs:0=0.5,2=0.5 ,3=0.5) {$E$};
	  \end{tikzpicture}
	  	  	  \begin{tikzpicture}[line join = round, line cap = round,scale=1.2, font=\footnotesize]
	  	
	  	\coordinate (3) at (0,{sqrt(2)},0);
	  	\coordinate (2) at ({-.5*sqrt(3)},0,-.5);
	  	\coordinate (1) at (0,0,1);
	  	\coordinate (0) at ({.5*sqrt(3)},0,-.5);
	  	
	  	\node at (3) {$\bullet$};
	  	\node at (2) {$\bullet$};
	  	
	  	\node at (0) {$\bullet$};
	  	
	  	\node[above] at (3) {(01111111)};
	  	
	  	\begin{scope}[decoration={markings,mark=at position 0.5 with {\arrow{to}}}]
	  		\draw[thick, postaction={decorate}] (2)--(0);
	  		\draw[fill=blue!10,fill opacity=.5] (1)--(0)--(3)--cycle;
	  		\draw[fill=blue!10,fill opacity=.5] (2)--(1)--(3)--cycle;
	  		\draw[thick, postaction={decorate}] (1)--(0) node[midway,below]{$0$};
	  		\draw[thick, postaction={decorate}] (2)--(1) node[midway,left]{$1$};
	  		\draw[thick, postaction={decorate}] (2)--(3);
	  		\draw[thick, postaction={decorate}] (1)--(3);
	  		\draw[thick, postaction={decorate}] (0)--(3) node[midway,right]{$1$};
	  	\end{scope}
	  	
	  	\node at (1) {$\bullet$};
	  	\node[opacity=.5] at (barycentric cs:0=0.5,1=0.5 ,2=0.5) {$E$};
	  	\node at (barycentric cs:0=0.5,1=0.5 ,3=0.5) {$E$};
	  	\node at (barycentric cs:1=0.2,2=0.5 ,3=0.4) {~$E$};
	  	\node[opacity=.5] at (barycentric cs:0=0.5,2=0.5 ,3=0.5) {$E$};
	  \end{tikzpicture}
	  
	  	  \begin{tikzpicture}[line join = round, line cap = round,scale=1.2,font=\footnotesize]
	  	
	  	\coordinate (3) at (0,{sqrt(2)},0);
	  	\coordinate (2) at ({-.5*sqrt(3)},0,-.5);
	  	\coordinate (1) at (0,0,1);
	  	\coordinate (0) at ({.5*sqrt(3)},0,-.5);
	  	
	  	\node at (3) {$\bullet$};
	  	\node at (2) {$\bullet$};
	  	
	  	\node[above] at (3) {(01111001)};
	  	
	  	\node at (0) {$\bullet$};
	  	
	  	\begin{scope}[decoration={markings,mark=at position 0.5 with {\arrow{to}}}]
	  		\draw[thick, postaction={decorate}] (2)--(0);
	  		\draw[fill=blue!10,fill opacity=.5] (1)--(0)--(3)--cycle;
	  		\draw[fill=blue!10,fill opacity=.5] (2)--(1)--(3)--cycle;
	  		\draw[thick, postaction={decorate}] (1)--(0) node[midway,below]{$1$};
	  		\draw[thick, postaction={decorate}] (2)--(1) node[midway,left]{$1$};
	  		\draw[thick, postaction={decorate}] (2)--(3);
	  		\draw[thick, postaction={decorate}] (1)--(3);
	  		\draw[thick, postaction={decorate}] (0)--(3) node[midway,right]{$1$};
	  	\end{scope}
	  	
	  	\node at (1) {$\bullet$};
	  	\node[opacity=.5] at (barycentric cs:0=0.5,1=0.5 ,2=0.5) {$E$};
	  	\node at (barycentric cs:0=0.5,1=0.5 ,3=0.5) {$D$};
	  	\node at (barycentric cs:1=0.2,2=0.6 ,3=0.5) {~$C$};
	  	\node[opacity=.5] at (barycentric cs:0=0.5,2=0.5 ,3=0.5) {$E$};
	  \end{tikzpicture}
	  \begin{tikzpicture}[line join = round, line cap = round,scale=1.2,font=\footnotesize]
	  	
	  	\coordinate (3) at (0,{sqrt(2)},0);
	  	\coordinate (2) at ({-.5*sqrt(3)},0,-.5);
	  	\coordinate (1) at (0,0,1);
	  	\coordinate (0) at ({.5*sqrt(3)},0,-.5);
	  	
	  	\node at (3) {$\bullet$};
	  	\node at (2) {$\bullet$};
	  	
	  	\node at (0) {$\bullet$};
	  	
	  	\node[above] at (3) {(00010111)};
	  	
	  	\begin{scope}[decoration={markings,mark=at position 0.5 with {\arrow{to}}}]
	  		\draw[thick, postaction={decorate}] (2)--(0);
	  		\draw[fill=blue!10,fill opacity=.5] (1)--(0)--(3)--cycle;
	  		\draw[fill=blue!10,fill opacity=.5] (2)--(1)--(3)--cycle;
	  		\draw[thick, postaction={decorate}] (1)--(0) node[midway,below]{$0$};
	  		\draw[thick, postaction={decorate}] (2)--(1) node[midway,left]{$0$};
	  		\draw[thick, postaction={decorate}] (2)--(3);
	  		\draw[thick, postaction={decorate}] (1)--(3);
	  		\draw[thick, postaction={decorate}] (0)--(3) node[midway,right]{$1$};
	  	\end{scope}
	  	
	  	\node at (1) {$\bullet$};
	  	\node[opacity=.5] at (barycentric cs:0=0.5,1=0.5 ,2=0.5) {$A$};
	  	\node at (barycentric cs:0=0.5,1=0.5 ,3=0.5) {$B$};
	  	\node at (barycentric cs:1=0.2,2=0.5 ,3=0.4) {~$B$};
	  	\node[opacity=.5] at (barycentric cs:0=0.5,2=0.5 ,3=0.5) {$B$};
	  \end{tikzpicture}
	  \begin{tikzpicture}[line join = round, line cap = round,scale=1.2, font=\footnotesize]
	  	
	  	\coordinate (3) at (0,{sqrt(2)},0);
	  	\coordinate (2) at ({-.5*sqrt(3)},0,-.5);
	  	\coordinate (1) at (0,0,1);
	  	\coordinate (0) at ({.5*sqrt(3)},0,-.5);
	  	
	  	\node at (3) {$\bullet$};
	  	\node at (2) {$\bullet$};
	  	
	  	\node at (0) {$\bullet$};
	  	
	  	\node[above] at (3) {(01111110)};
	  	
	  	\begin{scope}[decoration={markings,mark=at position 0.5 with {\arrow{to}}}]
	  		\draw[thick, postaction={decorate}] (2)--(0);
	  		\draw[fill=blue!10,fill opacity=.5] (1)--(0)--(3)--cycle;
	  		\draw[fill=blue!10,fill opacity=.5] (2)--(1)--(3)--cycle;
	  		\draw[thick, postaction={decorate}] (1)--(0) node[midway,below]{$1$};
	  		\draw[thick, postaction={decorate}] (2)--(1) node[midway,left]{$1$};
	  		\draw[thick, postaction={decorate}] (2)--(3);
	  		\draw[thick, postaction={decorate}] (1)--(3);
	  		\draw[thick, postaction={decorate}] (0)--(3) node[midway,right]{$1$};
	  	\end{scope}
	  	
	  	\node at (1) {$\bullet$};
	  	\node[opacity=.5] at (barycentric cs:0=0.5,1=0.5 ,2=0.5) {$E$};
	  	\node at (barycentric cs:0=0.5,1=0.5 ,3=0.5) {$E$};
	  	\node at (barycentric cs:1=0.2,2=0.5 ,3=0.4) {~$D$};
	  	\node[opacity=.5] at (barycentric cs:0=0.5,2=0.5 ,3=0.5) {$D$};
	  \end{tikzpicture}
	  \begin{tikzpicture}[line join = round, line cap = round,scale=1.2, font=\footnotesize]
	  	
	  	\coordinate (3) at (0,{sqrt(2)},0);
	  	\coordinate (2) at ({-.5*sqrt(3)},0,-.5);
	  	\coordinate (1) at (0,0,1);
	  	\coordinate (0) at ({.5*sqrt(3)},0,-.5);
	  	
	  	\node at (3) {$\bullet$};
	  	\node at (2) {$\bullet$};
	  	
	  	\node at (0) {$\bullet$};
	  	
	  	\node[above] at (3) {(01101111)};
	  	
	  	\begin{scope}[decoration={markings,mark=at position 0.5 with {\arrow{to}}}]
	  		\draw[thick, postaction={decorate}] (2)--(0);
	  		\draw[fill=blue!10,fill opacity=.5] (1)--(0)--(3)--cycle;
	  		\draw[fill=blue!10,fill opacity=.5] (2)--(1)--(3)--cycle;
	  		\draw[thick, postaction={decorate}] (1)--(0) node[midway,below]{$1$};
	  		\draw[thick, postaction={decorate}] (2)--(1) node[midway,left]{$1$};
	  		\draw[thick, postaction={decorate}] (2)--(3);
	  		\draw[thick, postaction={decorate}] (1)--(3);
	  		\draw[thick, postaction={decorate}] (0)--(3) node[midway,right]{$0$};
	  	\end{scope}
	  	
	  	\node at (1) {$\bullet$};
	  	\node[opacity=.5] at (barycentric cs:0=0.5,1=0.5 ,2=0.5) {$E$};
	  	\node at (barycentric cs:0=0.5,1=0.5 ,3=0.5) {$C$};
	  	\node at (barycentric cs:1=0.2,2=0.5 ,3=0.4) {~$E$};
	  	\node[opacity=.5] at (barycentric cs:0=0.5,2=0.5 ,3=0.5) {$C$};
	  \end{tikzpicture}
	  \end{center}

	  A 3-simplex can be represented by a binary 8-tuple $(x_0x_1x_2x_3x_4x_5x_6x_7)$ thought of as the function $\mathcal{P}(3)\to\KK$ with the domain ordered as \[(\varnothing,\{1\}, \{2\}, \{3\}, \{1,2\},\{1,3\},\{2,3\},\{1,2,3\})\]
	  All such 8-tuples correspond to valid Dynkin systems on $\{0,1,2,3\}$. However, the 8-tuples $(01110001)$, $(01111001)$, $(01110101)$ and $(01110011)$ correspond to the four Dynkin systems which are not intersection closed and therefore do not correspond to partitions. Equivalently, they are ``associators'' which are not admissible for the hyperoperation on $\Part(\ulp{1})$ described at the end of the last section.
      
      Noticing that the leading 0 is redundant in all 8-tuples (since $\varnothing\mapsto 0$ for any plasma morphism $\mathcal{P}(3)\to\KK$), we may also interpret each 3-simplex as a 7-tuple of the form \[(a_1,a_2,a_3,a_1+a_2,a_1+a_3,a_2+a_3,a_1+a_2+a_3)\] in which $a_i\in\KK$ for each $1\leq i\leq 3$ and the sums are also computed in $\KK$. However, these sums must satisfy relations forcing a version of associativity corresponding to ``ways to get from $(a_1,a_2,a_3)$ to $a_1+a_2+a_3$'' in $\KK$ (mirrored by their associated 3-simplices). The most important relation to be satisfied in such 7-tuples is the following:
      \[a_1+a_2+a_3\in ((a_1+a_2)+a_3))\cap ((a_1+a_3)+a_2)\cap (a_1+(a_2+a_3))\] where each binary sum is determined by the preceding entries of the 7-tuple. A complete description of the conditions that a 3-simplex or 4-simplex must satisfy can be found in \cite[Example 3.16]{beardsleynakamura}.

			\printbibliography
			
		\end{document}